\documentclass[reqno]{amsart}
\usepackage{amssymb, amsmath}
\newtheorem{theorem}{Theorem}

\newtheorem{corollary}[theorem]{Corollary}

\newtheorem{lemma}[theorem]{Lemma}

\newtheorem{proposition}[theorem]{Proposition}

\begin{document}

\title[Inverse Problem]{Necessary and Sufficient Conditions 
for Solvability of Inverse Problem for Dirac Operators with Discontinuous Coefficient}
\author[Kh. R. Mamedov and O. Akcay]{Khanlar R. Mamedov and Ozge Akcay$^{*}$}
\address{Science and Letters Faculty, Mathematics Department,
Mersin University, 33343, Turkey}
\email{hanlar@mersin.edu.tr and ozge.akcy@gmail.com}
\keywords{Dirac operator, eigenvalues and normalizing numbers, expansion formula, 
inverse problem, necessary and sufficient conditions}
\subjclass[2010]{34A55, 34L40.}
\thanks{$^{*}$corresponding author}
\date{}
\maketitle

\begin{abstract}
In this work, a complete solution of the inverse spectral problem for a class 
of Dirac differential equations system is given by spectral data (eigenvalues and normalizing numbers). 
As a direct problem, the eigenvalue problem is solved: the asymptotic formulas of eigenvalues, eigenfunctions and 
normalizing numbers of problem are obtained, spectral data is defined by the sets of eigenvalues 
and normalizing numbers. The expansion formula with respect to eigenfunctions is obtained.
Gelfand-Levitan-Marchenko equation is derived. 
The main theorem on necessary and sufficient conditions for the solvability of inverse spectral problem is 
proved and the algorithm of reconstruction of potential from spectral data is given.
\end{abstract}

\maketitle
%%% ----------------------------------------------------------------------
\section{Introduction}
We consider the boundary value problem generated by Dirac differential equations system
\begin{equation}\label{1}
By'+\Omega(x)y=\lambda\rho(x)y, \ \ \ 0<x<\pi
\end{equation}
with boundary condition
\begin{equation}\label{2}
\begin{array}{c}
U(y):= y_{1}(0)=0, \\
\\
V(y):= \left(\lambda+h_{1}\right)y_{1}(\pi)+h_{2}y_{2}(\pi)=0,
\end{array}
\end{equation}
where
\[
B=\left( 
\begin{array}{cc}
0 & 1 \\ 
-1 & 0
\end{array}
\right) ,\ \ \ \ \Omega \left(x\right) =\left( 
\begin{array}{cc}
p(x) & q(x) \\ 
q(x) & -p(x)
\end{array}
\right) ,\ \ \ \ y=\left( 
\begin{array}{c}
y_{1}\left( x\right) \\ 
y_{2}\left( x\right)
\end{array}
\right), 
\]
$p\left( x\right),$ $q\left( x\right) $ are real valued functions in $L_{2}(0,\pi),$ 
$\lambda$ is a spectral parameter, 
\[
\rho(x)=\left\{\begin{array}{c}
1, \ \ \ \ \ 0\leq x\leq a, \\
\alpha, \ \ \ \ a<x\leq\pi,
\end{array}
\right.
\]
$1\neq\alpha>0$, $h_{1}$ and $h_{2}$ are real numbers, $h_{2}>0.$

In the finite interval, in the case of $\rho(x)\equiv 1$ in the equation (\ref{1}) and 
the potential function $\Omega(x)$ is continuous, the solvabilty of 
inverse problem according to two spectra was obtained in \cite{Gas-Dz} (in this work
the criterion was obtained for two sequences of real numbers to be the spectra 
of two boundary value problems of Dirac operator) and according to one spectrum  
and normalizing numbers was given in \cite{Dz}. 
The inverse problem contained spectral parameter in boundary condition by spectral function was studied in \cite{MaS}. 
Inverse spectral problems for Dirac operator with summable potential were worked in \cite{Al-Hr-My, My-Pu, Pu}. 
Reconstruction of Dirac operator from nodal data was carried out in \cite{Ch-Zh}. 
For Dirac operator with peculiarity inverse problem was found out in \cite{Pa}.
Using Weyl-Titschmarsh function, direct and inverse problems for Dirac type-system were studied in \cite{Sa3, Sa2, Sa1}. 
Solution of the inverse quasiperiodic problem for Dirac system was given in \cite{Na}.
Inverse problem for weighted Dirac equations was investigated in \cite{Wa}. 
For Dirac operator Ambarzumian-type theorems were proved in \cite{Ho, Ki, Ch-X}.
On a positive half line, inverse scattering problem for a system of Dirac
equations of order $2n$ was completely solved in \cite{Ga} and when boundary condition
contained spectral parameter, for Dirac operator, inverse scattering
problem was worked in \cite{Mam1, Mam2}.
The applications of Dirac differential equations
system has been widespread in various areas of physics, such as \cite{Ok1, Ok2, Sh2, Sh1, Th} 
since Dirac equation was discovered to be associated with nonlinear wave equation in \cite{Ab1}.

In this paper, as different from other studies, the boundary value problem (\ref{1}), (\ref{2}) 
has piecewise continuous coefficient, so the integral representation (not operator transformation)
for the solution of equation (\ref{1}) obtained in \cite{La1} is used.

This paper is organized as follows: In section 2, based on this integral representation, the asymptotic formulas of 
eigenvalues, eigenfunctions and normalizing numbers of the boundary value problem (\ref{1}), (\ref{2}) are investigated.
The completeness theorem with respect to eigenfunctions is proved. 
The spectral expansion formula is obtained and Parseval equality is given. In section 3, the main equation 
namely Gelfand-Levitan-Marchenko type equation is derived. In section 4, 
we show that the boundary value problem (\ref{1}), (\ref{2}) can be uniquely determined from its 
eigenvalues and normalizing numbers. Finally in section 5, the solution of inverse problem is obtained. 
Let's express this more clearly. We can state the inverse problem for a system of Dirac equations in the following way: 
let $\lambda_{n}$ and $\alpha_{n}$, $\left(n\in\mathbb{Z}\right)$ are respectively eigenvalues and normalizing numbers of boundary value problem (\ref{1}), (\ref{2}). Knowing the spectral data   
$\left\{\lambda_{n},\alpha_{n}\right\}$, $\left(n\in\mathbb{Z}\right)$ 
to indicate a method of determining the potential $\Omega(x)$ and to find necessary and sufficient 
conditions for $\left\{\lambda_{n},\alpha_{n}\right\}$, 
$\left(n\in\mathbb{Z}\right)$ to be the spectral data of a problem (\ref{1}),(\ref{2}), for this, we derive 
differential equation, Parseval equality and boundary conditions. The main theorem on the necessary and sufficient conditions
for the solvability of inverse problem is proved and then algorithm of the construction of the function $\Omega(x)$
by spectral data is given. Note that throughout this paper, we use the following notation: $\tilde{\phi}$ denotes
the transposed matrix of $\phi$.

\section{Preliminaries}
The inner product in Hilbert space $H=L_{2,\rho}(0,\pi;\mathbb{C}^{2})\oplus \mathbb{C}$ is defined by
\[
\left\langle Y,Z\right\rangle:=\int_{0}^{\pi}\left[y_{1}(x)\overline{z}_{1}(x)+
y_{2}(x)\overline{z}_{2}(x)\right]\rho(x)dx+\frac{1}{h_{2}}y_{3}\overline{z}_{3},
\] where 
\[
Y=\left( 
\begin{array}{c}
y_{1}(x) \\ 
y_{2}(x) \\ 
y_{3}
\end{array}
\right) \in H, \ \ \ \ \ \ \ \ Z=\left( 
\begin{array}{c}
z_{1}(x) \\ 
z_{2}(x) \\ 
z_{3}
\end{array}
\right) \in H. 
\]

Let us define 
\[
L(Y):=\left( 
\begin{array}{c}
l(y) \\ 
-h_{1}y_{1}(\pi)-h_{2}y_{2}(\pi)
\end{array}
\right) 
\]
with
\[
D(L)=\left\{ 
\begin{array}{c}
Y\left|\right.Y=\left( y_{1}(x),y_{2}(x),y_{3}\right)^{T} \in H,\
y_{1}(x),\ y_{2}(x)\in AC[0,\pi ], \\ 
y_{3}=y_{1}(\pi) ,\ l(y)\in L_{2,\rho}(0,\pi; \mathbb{C}^{2}),\
y_{1}(0)=0
\end{array}
\right\} 
\]
where
\[
l(y)=\frac{1}{\rho(x)}\left(\begin{array}{c}
y_{2}^{'}+p(x)y_{1}+q(x)y_{2}\\
-y_{1}^{'}+q(x)y_{1}-p(x)y_{2}
\end{array}\right) . 
\]
The boundary value problem (\ref{1}), (\ref{2}) is equivalent to equation $LY=\lambda Y.$

Let $\varphi(x,\lambda)$ and $\psi(x,\lambda)$ be solutions of the system (\ref{1}) satisfying
the initial conditions
\begin{equation}\label{b}
\varphi(0,\lambda)=\left(\begin{array}{c}
0 \\
-1
\end{array}\right), \ \ \ \psi(\pi,\lambda)=\left(
\begin{array}{c}
h_{2} \\
-\lambda-h_{1}
\end{array}
\right). 
\end{equation} 
The solution $\varphi(x,\lambda)$ has the following representation \cite{La1, Mam-Ak}
\begin{equation}\label{3}
\varphi(x,\lambda)=\varphi_{0}(x,\lambda)+\int_{0}^{\mu(x)}A(x,t)\left(\begin{array}{c}
\sin\lambda t \\
-\cos\lambda t
\end{array}
\right)dt,
\end{equation}
where 
\[\varphi_{0}(x,\lambda)=\left(\begin{array}{c}
\sin\lambda\mu(x) \\
-\cos\lambda\mu(x)
\end{array}\right), \ \ \ \ \mu(x)=\left\{\begin{array}{c}
x, \ \ \ \ \ \ \ \ \ \ \ \ \ \ \ \ \  0\leq x\leq a, \\
\alpha x-\alpha a+a, \ \ \ \ a<x\leq\pi,
\end{array}\right.
\]
$\left(A_{ij}\right)^{2}_{i,j=1}$ is quadratic matrix function, $A_{ij}(x,.)\in L_{2}(0,\pi), \ i,j=1,2$ 
for fixed $x\in[0,\pi]$ and $A(x,t)$
is solution of the problem
\[
BA'_{x}(x,t)+\rho(t)A'_{t}(x,t)B=-\Omega(x)A(x,t),
\]
\begin{equation}\label{4}
\Omega(x)=\rho(x)A(x,\mu(x))B-BA(x,\mu(x)),
\end{equation}
\[
A_{11}(x,0)=A_{21}(x,0)=0.
\] The formula (\ref{4}) gives the relation between the kernel $A(x,t)$
and the coefficient of $\Omega(x)$ of the equation (\ref{1}). 

The characteristic function $\Delta(\lambda)$ of the problem $L$ is 
\begin{equation}\label{5}
\Delta(\lambda):=W\left[\varphi(x,\lambda), \psi(x,\lambda)\right]=
\varphi_{2}(x,\lambda)\psi_{1}(x,\lambda)-\varphi_{1}(x,\lambda)\psi_{2}(x,\lambda),
\end{equation} 
where $W\left[\varphi(x,\lambda), \psi(x,\lambda)\right]$ is Wronskian of the solutions $\varphi(x,\lambda)$ and 
$\psi(x,\lambda)$ and independent of $x\in[0,\pi].$
The zeros of $\Delta(\lambda)$ coincide with the eigenvalues $\lambda_{n}$ of problem $L$. 
The functions $\varphi(x,\lambda)$ and $\psi(x,\lambda)$ are eigenfunctions and there exists a sequence $\beta_{n}$
such that
\begin{equation}\label{6}
\psi(x,\lambda_{n})=\beta_{n}\varphi(x,\lambda_{n}), \ \ \ \beta_{n}\neq 0.
\end{equation}

Normalizing numbers are
\[
\alpha_{n}:=\int_{0}^{\pi}\left(\left|\varphi_{1}(x,\lambda_{n})\right|^{2}+
\left|\varphi_{2}(x,\lambda_{n})\right|^{2}\right)\rho(x)dx+
\frac{1}{h_{2}}\left|\varphi_{1}(\pi,\lambda_{n})\right|^{2}.
\] 
The following relation holds \cite{Mam-Ak}:
\begin{equation}\label{7}
\dot{\Delta}(\lambda_{n})=\beta_{n}\alpha_{n},
\end{equation}
where $\dot{\Delta}(\lambda)=\frac{d}{d\lambda}\Delta(\lambda)$. 

\begin{theorem} \label{thm1}
i) The eigenvalues $\lambda_{n},$ $\left(n\in\mathbb{Z}\right)$ of boundary value problem (\ref{1}), (\ref{2}) are 
\begin{equation}\label{8}
\lambda_{n}=\lambda_{n}^{0}+\epsilon_{n}, \ \ \ \left\{\epsilon_{n}\right\}\in l_{2},
\end{equation}
where $\lambda_{n}^{0}=\frac{n\pi}{\mu(\pi)}$ are zeros of function $\lambda\sin\lambda\mu(\pi)$.
For the large $n$, the eigenvalues are simple. \\

ii) The eigenfunctions of the boundary value problem can be represented in the form
\begin{equation}\label{9}
\varphi(x,\lambda_{n})=\left(\begin{array}{c}
\sin\frac{n\pi\mu(x)}{\mu(\pi)} \\
-\cos\frac{n\pi\mu(x)}{\mu(\pi)}
\end{array}
\right)+\left(\begin{array}{c}
\zeta_{n}^{(1)}(x) \\
\zeta_{n}^{(2)}(x) 
\end{array}
\right),
\end{equation}
where $\sum_{n=-\infty}^{\infty}\left\{\left|\zeta_{n}^{(1)}(x)\right|^{2}+\left|
\zeta_{n}^{(2)}(x)\right|^{2}\right\}\leq C,$ in here $C$ is a positive number. \\

iii)Normalizing numbers of the problem (\ref{1}), (\ref{2}) are as follows
\begin{equation}\label{10}
\alpha_{n}=\mu(\pi)+\tau_{n}, \ \ \ \left\{\tau_{n}\right\}\in l_{2}.
\end{equation}
\end{theorem}
\begin{proof}
The proof of this theorem is similarly obtained in \cite{Mam-Ak}.
\end{proof}

We note that using (\ref{3}), as $\left|\lambda\right|\rightarrow\infty$ uniformly in $x\in[0,\pi]$ the following asymptotic 
formulas are obtain:
\begin{equation}\label{11}
\begin{array}{c}
\varphi_{1}(x,\lambda)=\sin\lambda\mu(x)+O\left(\frac{1}{\left|\lambda\right|}e^{\left|Im\lambda\right|\mu(x)}\right), \\
\\
\varphi_{2}(x,\lambda)=-\cos\lambda\mu(x)+O\left(\frac{1}{\left|\lambda\right|}e^{\left|Im\lambda\right|\mu(x)}\right).
\end{array}
\end{equation}
Substituting the asymptotic formulas (\ref{11}) into
\[
\Delta(\lambda)=\left(\lambda+h_{1}\right)\varphi_{1}(\pi,\lambda)+h_{2}\varphi_{2}(\pi,\lambda),
\] we get as $\left|\lambda\right|\rightarrow\infty$
\begin{equation}\label{12}
\Delta(\lambda)=\lambda\sin\lambda\mu(\pi)+O\left(e^{\left|Im\lambda\right|\mu(\pi)}\right).
\end{equation}
\begin{proposition} \label{p1}
The specification of the eigenvalues $\lambda_{n}, \ (n\in\mathbb{Z})$ uniquely determines the
characteristic function $\Delta(\lambda)$ by formula 
\begin{equation}\label{a}
\Delta(\lambda)=-\mu(\pi)(\lambda_{0}^{2}-\lambda^{2})\prod_{n=1}^{\infty}
\frac{(\lambda_{n}^{2}-\lambda^{2})}{(\lambda_{n}^{0})^{2}}.
\end{equation}
\end{proposition}
\begin{proof}
Since the function $\Delta(\lambda)$ is entire function, from Hadamard's theorem (see \cite{Le}),
using (\ref{12}) we obtain (\ref{a}).
\end{proof}
\begin{theorem} \label{thm2}
a) The system of eigenfunctions $\varphi(x,\lambda_{n})$, $(n\in \mathbb{Z})$ of boundary value problem (\ref{1}), (\ref{2}) 
is complete in space $L_{2,\rho}(0,\pi;\mathbb{C}^{2})$. \\

b) Let $f(x)\in D(L).$  Then the expansion formula
\begin{equation}\label{16}
f(x)=\sum_{n=-\infty}^{\infty}a_{n}\varphi(x,\lambda_{n}),
\end{equation}
\[
a_{n}=\frac{1}{\alpha_{n}}\left\langle f(x), \varphi(x,\lambda_{n})\right\rangle
\] is valid and the series converges uniformly with respect to $x\in[0,\pi].$ 
For $f(x)\in L_{2,\rho}(0,\pi;\mathbb{C}^{2})$ series (\ref{16}) converges in $L_{2,\rho}(0,\pi;\mathbb{C}^{2})$, moreover,
Parseval equality holds
\begin{equation}\label{17}
\left\|f\right\|^{2}=\sum_{n=-\infty}^{\infty}\alpha_{n}\left|a_{n}\right|^{2}.
\end{equation}
\end{theorem}
\begin{proof}
This theorem is analogously proved in \cite{Mam-Ak}.  
\end{proof}

\section{Main Equation}
\begin{theorem} \label{thm3}
For each fixed $x\in(0,\pi]$ the kernel $A(x,t)$ from the representation (\ref{3}) satisfies the following equation
\begin{equation}\label{28}
A(x,\mu(t))+F(x,t)+\int_{0}^{\mu(x)}A(x,\xi)F_{0}(\xi,t)d\xi=0, \ \ \ \ 0<t<x,
\end{equation}
where
\begin{equation}\label{29}
F_{0}(x,t)=\sum_{n=-\infty}^{\infty}\left[\frac{1}{\alpha_{n}}\left(\begin{array}{c}
\sin\lambda_{n}x \\
-\cos\lambda_{n}x
\end{array}\right)\tilde{\varphi}_{0}(t,\lambda_{n})-
\frac{1}{\mu(\pi)}\left(\begin{array}{c}
\sin\lambda_{n}^{0}x \\
-\cos\lambda_{n}^{0}x
\end{array}\right)\tilde{\varphi}_{0}(t,\lambda_{n}^{0})\right]
\end{equation}
and
\begin{equation} \label{30}
F(x,t)=F_{0}(\mu(x),t).
\end{equation} 
\end{theorem}
\begin{proof}
According to (\ref{3}) we have, 
\begin{equation}\label{31}
\varphi_{0}(x,\lambda)=\varphi(x,\lambda)-\int_{0}^{\mu(x)}A(x,t)\left(\begin{array}{c}
\sin\lambda t \\
-\cos\lambda t
\end{array}\right)dt.
\end{equation}
It follows from (\ref{3}) and (\ref{31}) that
\[
\sum_{n=-N}^{N}\frac{1}{\alpha_{n}}\varphi(x,\lambda_{n})\tilde{\varphi_{0}}(t,\lambda_{n})=
\sum_{n=-N}^{N}\frac{1}{\alpha_{n}}\varphi_{0}(x,\lambda_{n})\tilde{\varphi_{0}}(t,\lambda_{n})+
\]
\[
+\int_{0}^{\mu(x)}A(x,\xi)\left(\sum_{n=-N}^{N}\frac{1}{\alpha_{n}}\left(\begin{array}{c}
\sin\lambda_{n}\xi \\
-\cos\lambda_{n}\xi
\end{array}\right)\tilde{\varphi_{0}}(t,\lambda_{n})\right)d\xi
\]
and 
\[
\sum_{n=-N}^{N}\frac{1}{\alpha_{n}}\varphi(x,\lambda_{n})\tilde{\varphi_{0}}(t,\lambda_{n})=
\sum_{n=-N}^{N}\frac{1}{\alpha_{n}}\varphi(x,\lambda_{n})\tilde{\varphi}(t,\lambda_{n})-
\]
\[
-\sum_{n=-N}^{N}\frac{1}{\alpha_{n}}\varphi(x,\lambda_{n})
\int_{0}^{\mu(t)}\left(\sin\lambda_{n}\xi, \ -\cos\lambda_{n}\xi\right)\tilde{A}(t,\xi)d\xi.
\]
Using the last two equalities, we obtain 
\[
\sum_{n=-N}^{N}\left[\frac{1}{\alpha_{n}}\varphi(x,\lambda_{n})\tilde{\varphi}(t,\lambda_{n})-
\frac{1}{\mu(\pi)}\varphi(x,\lambda_{n}^{0})\tilde{\varphi}(t,\lambda_{n}^{0})\right]=
\]
\[
=\sum_{n=-N}^{N}\left[\frac{1}{\alpha_{n}}
\varphi_{0}(x,\lambda_{n})\tilde{\varphi_{0}}(t,\lambda_{n})-\frac{1}{\mu(\pi)}
\varphi_{0}(x,\lambda_{n}^{0})\tilde{\varphi_{0}}(t,\lambda_{n}^{0})\right]+
\]
\[
+\int_{0}^{\mu(x)}A(x,\xi)\sum_{n=-N}^{N}\left[\frac{1}{\mu(\pi)}\left(\begin{array}{c}
\sin\lambda_{n}^{0}\xi \\
-\cos\lambda_{n}^{0}\xi
\end{array}\right)\tilde{\varphi_{0}}(t,\lambda_{n}^{0})\right]d\xi+
\]
\begin{eqnarray}\nonumber
\lefteqn{+\int_{0}^{\mu(x)}A(x,\xi)\sum_{n=-N}^{N}\left[\frac{1}{\alpha_{n}}\left(\begin{array}{c}
\sin\lambda_{n}\xi \\
-\cos\lambda_{n}\xi
\end{array}\right)\tilde{\varphi_{0}}(t,\lambda_{n})- \right.} \\ \nonumber
&&\left.-\frac{1}{\mu(\pi)}\left(\begin{array}{c}
\sin\lambda_{n}^{0}\xi \\
-\cos\lambda_{n}^{0}\xi
\end{array}\right)\tilde{\varphi_{0}}(t,\lambda_{n}^{0})\right]d\xi+ \nonumber
\end{eqnarray}
\[
+\sum_{n=-N}^{N}\frac{1}{\alpha_{n}}\varphi(x,\lambda_{n})
\int_{0}^{\mu(t)}\left(\sin\lambda_{n}\xi, \ -\cos\lambda_{n}\xi\right)\tilde{A}(t,\xi)d\xi
\]
or 
\begin{equation}\label{32}
\Phi_{N}(x,t)=I_{N1}(x,t)+I_{N2}(x,t)+I_{N3}(x,t)+I_{N4}(x,t),
\end{equation} where
\[
\Phi_{N}(x,t)=\sum_{n=-N}^{N}\left[\frac{1}{\alpha_{n}}\varphi(x,\lambda_{n})\tilde{\varphi}(t,\lambda_{n})-
\frac{1}{\mu(\pi)}\varphi(x,\lambda_{n}^{0})\tilde{\varphi}(t,\lambda_{n}^{0})\right],
\]
\[
I_{N1}(x,t)=\sum_{n=-N}^{N}\left[\frac{1}{\alpha_{n}}
\varphi_{0}(x,\lambda_{n})\tilde{\varphi_{0}}(t,\lambda_{n})-\frac{1}{\mu(\pi)}
\varphi_{0}(x,\lambda_{n}^{0})\tilde{\varphi_{0}}(t,\lambda_{n}^{0})\right],
\]
\[
I_{N2}(x,t)=\int_{0}^{\mu(x)}A(x,\xi)\sum_{n=-N}^{N}\left[\frac{1}{\mu(\pi)}\left(\begin{array}{c}
\sin\lambda_{n}^{0}\xi \\
-\cos\lambda_{n}^{0}\xi
\end{array}\right)\tilde{\varphi_{0}}(t,\lambda_{n}^{0})\right]d\xi,
\]
\begin{eqnarray}\nonumber
I_{N3}(x,t)&=&\int_{0}^{\mu(x)}A(x,\xi)\sum_{n=-N}^{N}\left[\frac{1}{\alpha_{n}}\left(\begin{array}{c}
\sin\lambda_{n}\xi \\
-\cos\lambda_{n}\xi
\end{array}\right)\tilde{\varphi_{0}}(t,\lambda_{n})-\right. \\ \nonumber
&&-\left.\frac{1}{\mu(\pi)}\left(\begin{array}{c}
\sin\lambda_{n}^{0}\xi \\
-\cos\lambda_{n}^{0}\xi
\end{array}\right)\tilde{\varphi_{0}}(t,\lambda_{n}^{0})\right]d\xi, \nonumber
\end{eqnarray}
\[
I_{N4}(x,t)=\sum_{n=-N}^{N}\frac{1}{\alpha_{n}}\varphi(x,\lambda_{n})
\int_{0}^{\mu(t)}\left(\sin\lambda_{n}\xi, \ -\cos\lambda_{n}\xi\right)\tilde{A}(t,\xi)d\xi.
\] It is easily found by using (\ref{29}) and (\ref{30}) that
\begin{equation}\label{33}
F(x,t)=\sum_{n=-\infty}^{\infty}\left[\frac{1}{\alpha_{n}}\varphi_{0}(x,\lambda_{n})\tilde{\varphi}_{0}(t,\lambda_{n})
-\frac{1}{\mu(\pi)}\varphi_{0}(x,\lambda_{n}^{0})\tilde{\varphi}_{0}(t,\lambda_{n}^{0})\right].
\end{equation}
Let $f(x)\in AC[0,\pi]$. Then according to expansion formula (\ref{16}) in Theorem \ref{thm2}, we obtain uniformly 
on $x\in[0,\pi]$
\begin{equation}\label{34}
\lim_{N\rightarrow\infty}\int_{0}^{\pi}\Phi_{N}(x,t)f(t)\rho(t)dt=
\sum_{n=-\infty}^{\infty}a_{n}\varphi(x,\lambda_{n})-\sum_{n=-\infty}^{\infty}a_{n}^{0}\varphi(x,\lambda_{n}^{0})=0.
\end{equation}
From (\ref{33}), we find
\[
\lim_{N\rightarrow\infty}\int_{0}^{\pi}I_{N1}(x,t)f(t)\rho(t)dt=
\]
\[
=\lim_{N\rightarrow\infty}
\int_{0}^{\pi}\sum_{n=-N}^{N}\left[\frac{1}{\alpha_{n}}
\varphi_{0}(x,\lambda_{n})\tilde{\varphi_{0}}(t,\lambda_{n})-
\frac{1}{\mu(\pi)}
\varphi_{0}(x,\lambda_{n}^{0})\tilde{\varphi_{0}}(t,\lambda_{n}^{0})\right]f(t)\rho(t)dt
\]
\begin{equation} \label{35}
=\int_{0}^{\pi}F(x,t)f(t)\rho(t)dt.  
\end{equation}
It follows from (\ref{3}) that 
\begin{equation} \label{36}
\left(\begin{array}{c}
\sin\lambda\xi \\
-\cos\lambda\xi
\end{array}\right)=\left\{ 
\begin{array}{c}
\varphi_{0}(\xi,\lambda),\ \ \ \ \ \ \ \ \ \ \ \ \ \ \ \ \ \xi< a, \\ 
\varphi_{0}\left(\frac{\xi}{\alpha}+a-\frac{a}{\alpha},\lambda\right) ,\ \ \ \xi> a.
\end{array}
\right.
\end{equation}
Taking into account (\ref{36}) and expansion formula (\ref{16}) in Theorem \ref{thm2} , we get
\[
\lim_{N\rightarrow\infty}\int_{0}^{\pi}I_{N2}(x,t)f(t)\rho(t)dt=  \ \ \ \ \ \ \ \ \ \ \ \ \ \ \ \ \ \ \ \ \ \ \ \ \ \ \ \ \ \
 \ \ \ \ \ \ \ \ \ \ \ \ \ \ \ \ \ \ \ \ \ \ \ \
\]
\[
=\int_{0}^{\pi}\left[\int_{0}^{\mu(x)}A(x,\xi)\sum_{n=-N}^{N}\left[\frac{1}{\mu(\pi)}\left(\begin{array}{c}
\sin\lambda_{n}^{0}\xi \\
-\cos\lambda_{n}^{0}\xi
\end{array}\right)\tilde{\varphi_{0}}(t,\lambda_{n}^{0})\right]d\xi\right]f(t)\rho(t)dt
\]
\[
=\int_{0}^{\pi}\left[\int_{0}^{a}A(x,\xi)\sum_{n=-\infty}^{\infty}\frac{1}{\mu(\pi)}\varphi_{0}(\xi,\lambda_{n}^{0})
\tilde{\varphi_{0}}(x,\lambda_{n}^{0})d\xi\right]f(t)\rho(t)dt+  \ \ \ \ \ \ \ \ \ \ \ \
\]
\[
+\int_{0}^{\pi}\left[\int_{a}^{\alpha x-\alpha a+a}A(x,\xi)
\sum_{n=-\infty}^{\infty}\frac{1}{\mu(\pi)}
\varphi_{0}\left(\frac{\xi}{\alpha}+a-\frac{a}{\alpha},
\lambda_{n}^{0}\right)
\tilde{\varphi_{0}}(x,\lambda_{n}^{0})d\xi\right]f(t)\rho(t)dt
\]
\[
=\int_{0}^{a}A(x,\xi)f(\xi)d\xi+\int_{a}^{\alpha x-\alpha a+a}A(x,\xi)f\left(\frac{\xi}{\alpha}+a-\frac{a}{\alpha}\right)d\xi.
\]
Substituting $\frac{\xi}{\alpha}+a-\frac{a}{\alpha}\rightarrow\xi'$, we obtain
\begin{eqnarray}\nonumber
\lefteqn{\lim_{N\rightarrow\infty}\int_{0}^{\pi}I_{N2}(x,t)f(t)\rho(t)dt=}
\\ \nonumber
&=&\int_{0}^{a}A(x,\xi)f(\xi)d\xi+\alpha\int_{a}^{x}A(x,\alpha\xi'-\alpha a+a)f(\xi')d\xi'
\\ \nonumber
&=&\int_{0}^{a}A(x,t)f(t)dt+\alpha\int_{a}^{x}A(x,\alpha t-\alpha a+a)f(t)dt \nonumber
\end{eqnarray}
\begin{equation}\label{37}
=\int_{0}^{x}A(x,\mu(t))f(t)\rho(t)dt.  \ \ \ \ \ \ \ \ \ \ \ \ \ \ \ \ \ \ \ \ \ \ \ \ \ \ \ \ \ \ \ \ \ \ \
\end{equation}
Now, we calculate
\begin{eqnarray}\nonumber
\lefteqn{\lim_{N\rightarrow\infty}\int_{0}^{\pi}I_{N3}(x,t)f(t)\rho(t)dt=}
\\ \nonumber
&=&\lim_{N\rightarrow\infty}\int_{0}^{\pi}\int_{0}^{\mu(x)}A(x,\xi)
\sum_{n=-N}^{N}\left[\frac{1}{\alpha_{n}}\left(\begin{array}{c}
\sin\lambda_{n}\xi \\
-\cos\lambda_{n}\xi
\end{array}\right)\tilde{\varphi_{0}}(t,\lambda_{n})-\right. \nonumber \\
&&\left.-\frac{1}{\mu(\pi)}\left(\begin{array}{c}
\sin\lambda_{n}^{0}\xi \\
-\cos\lambda_{n}^{0}\xi
\end{array}\right)\tilde{\varphi_{0}}(t,\lambda_{n}^{0})\right]f(t)\rho(t)d\xi dt \nonumber
\end{eqnarray}
\begin{equation}\label{38}
=\int_{0}^{\pi}\left[\int_{0}^{\mu(x)}A(x,\xi)F_{0}(\xi,t)d\xi\right]f(t)\rho(t)dt. \ \ \ \ \ \ \ \ \ \ \ \ \ \ \ \ 
\ \ \ \ \ \
\end{equation}
Using (\ref{6}), (\ref{7}) and residue theorem, we get
\[
\lim_{N\rightarrow\infty}\int_{0}^{\pi}I_{N4}(x,t)f(t)\rho(t)dt= \ \ \ \ \ \ \ \ \ \ \ \ \ \ \ \ \ \ \ \ \ \ \ \ 
\ \ \ \ \ \ \ \ \ \ \ \ \ \ \ \ 
\]
\[
=\lim_{N\rightarrow\infty}\int_{0}^{\pi}\left[\sum_{n=-N}^{N}\frac{1}{\alpha_{n}}\varphi(x,\lambda_{n})
\int_{0}^{\mu(t)}\left(\sin\lambda_{n}\xi, -\cos\lambda_{n}\xi\right)\tilde{A}(t,\xi)d\xi\right]f(t)\rho(t)dt
\]
\[
=\lim_{N\rightarrow\infty}\int_{0}^{\pi}
\left[\sum_{n=-N}^{N}\frac{\psi(x,\lambda_{n})}{\dot{\Delta}(\lambda_{n})}
\int_{0}^{\mu(t)}\left(\sin\lambda_{n}\xi, -\cos\lambda_{n}\xi\right)\tilde{A}(t,\xi)d\xi\right]f(t)\rho(t)dt \ \ \ 
\]
\[
=\lim_{N\rightarrow\infty}\int_{0}^{\pi}\left[\sum_{n=-N}^{N}\underset{\lambda =\lambda _{n}}{Res}
\frac{\psi(x,\lambda)}{\Delta(\lambda)}\int_{0}^{\mu(t)}
\left(\sin\lambda\xi, -\cos\lambda\xi\right)\tilde{A}(t,\xi)d\xi\right]f(t)\rho(t)dt \ \ 
\]
\[
=\lim_{N\rightarrow\infty}\int_{0}^{\pi}\left[\frac{1}{2\pi i}\int_{\Gamma_{N}}
\frac{\psi(x,\lambda)}{\Delta(\lambda)}\int_{0}^{\mu(t)}
\left(\sin\lambda\xi, \ -\cos\lambda\xi\right)\tilde{A}(t,\xi)d\xi d\lambda\right]f(t)\rho(t)dt
\]
\[
=\lim_{N\rightarrow\infty}\int_{0}^{\pi}\left[\frac{1}{2\pi i}\int_{\Gamma_{N}}
\frac{\psi(x,\lambda)}{\Delta(\lambda)}e^{\left|Im\lambda\right|\mu(t)}\times\right.
\ \ \ \ \ \ \ \ \ \ \ \ \ \ \ \ \ \ \ \ \ \ \ \ \ \ \ \ \ \ \ \ \ \ \ \ \ \ \ \ \ \ \ \ \ \ \
\]
\begin{equation}\label{39}
\left.\times e^{-\left|Im\lambda\right|\mu(t)}\int_{0}^{\mu(t)}
\left(\sin\lambda\xi, \ -\cos\lambda\xi\right)\tilde{A}(t,\xi)d\xi d\lambda\right]f(t)\rho(t)dt,
\end{equation}
where $\Gamma_{N}=\left\{\lambda:\left|\lambda\right|=\lambda_{N}^{0}+
\frac{\pi}{2\mu(\pi)}\right\}$ is a oriented counter-clockwise, $N$ is sufficiently
large number. Taking into account, the asymptotic formulas as $\left|\lambda\right|\rightarrow\infty$
\[
\psi_{1}(x,\lambda)=h_{2}\cos\lambda(\mu(\pi)-\mu(x))
-(\lambda+h_{1})\sin\lambda(\mu(\pi)-\mu(x))+O\left(e^{\left|Im\lambda\right|(\mu(\pi)-\mu(x))}\right), 
\]
\[
\psi_{2}(x,\lambda)=-h_{2}\sin\lambda(\mu(\pi)-\mu(x))
-(\lambda+h_{1})\cos\lambda(\mu(\pi)-\mu(x))+O\left(e^{\left|Im\lambda\right|(\mu(\pi)-\mu(x))}\right)
\]
and the relations (\cite{Mar}, Lemma 1.3.1)
\[
\lim_{\left|\lambda\right|\rightarrow\infty}\max_{0\leq t\leq\pi}e^{-\left|Im\lambda\right|\mu(t)}
\left|\int_{0}^{\mu(t)}A_{i,1}(t,\xi)\sin\lambda\xi d\xi\right|=0,
\]
\[
\lim_{\left|\lambda\right|\rightarrow\infty}\max_{0\leq t\leq\pi}e^{-\left|Im\lambda\right|\mu(t)}
\left|\int_{0}^{\mu(t)}A_{i,2}(t,\xi)\cos\lambda\xi d\xi\right|=0, \ \ i=1,2,
\] 
it follows from (\ref{8}) and (\ref{39}) that
\begin{equation}\label{40}
\lim_{N\rightarrow\infty}\int_{0}^{\pi}I_{N4}(x,t)f(t)\rho(t)dt=0.
\end{equation}
Thus, using (\ref{32}), (\ref{34}), (\ref{35}), (\ref{37}) (\ref{38}) and (\ref{40}), we find
\[
\int_{0}^{x}A(x,\mu(t))f(t)\rho(t)dt+\int_{0}^{\pi}F(x,t)f(t)\rho(t)dt+
\]
\[
+\int_{0}^{\pi}\left[\int_{0}^{\mu(x)}A(x,\xi)F_{0}(\xi,t)d\xi \right]f(t)\rho(t)dt=0.
\] Since $f(x)$ can be chosen arbitrarily,
\[
A(x,\mu(t))+F(x,t)+\int_{0}^{\mu(x)}A(x,\xi)F_{0}(\xi,t)d\xi=0, \ \ \ 0<t<x
\] is obtained.
\end{proof}

\section{Theorem for the Solution of the Inverse Problem}
\begin{lemma} \label{thm4}
For each fixed $x\in (0,\pi]$ the equation (\ref{28}) has a unique solution $A(x,.)\in L_{2}(0,\mu(x)).$
\end{lemma}
\begin{proof}
When $a<x$, the equation (\ref{28}) can be rewritten as
\[
L_{x}A(x,.)+K_{x}A(x,.)=-F(x,.),
\] where
\begin{equation}\label{41}
\left(L_{x}f\right)(t)=\left\{
\begin{array}{c}
f(t), \ \ \ \ \ \ \ \ \ \ \ \ \ \ \ \ t\leq a<x, \\
f(\alpha t-\alpha a+a), \ \ \ a<t\leq x,
\end{array}
\right.
\end{equation}
\[
\left(K_{x}f\right)=\int_{0}^{\alpha x-\alpha a+a}f(\xi)F_{0}(\xi,t)d\xi, \ \ 0<t<x.
\]
Now, we shall prove that $L_{x}$ is invertible, i.e has a bounded inverse in $L_{2}(0,\pi)$.

Consider the equation $\left(L_{x}f\right)(t)=\phi(t),$ $\phi(t)\in L_{2}(0,\pi;\mathbb{C}^{2}).$ 
From here and (\ref{41}),
\[
f(t)=\left(L_{x}^{-1}\phi\right)(t)=\left\{\begin{array}{c}
\phi(t), \ \ \ \ \ \ \ \ \ \ \ \ \ t\leq a, \\
\phi\left(\frac{t+\alpha a-a}{\alpha}\right), \ \ \ a<t.
\end{array}
\right.
\] 
We show that 
\[
\left\|f\right\|_{L_{2}}=\left\|L_{x}^{-1}\phi\right\|_{L_{2}}\leq C\left\|\phi\right\|_{L_{2}}. 
\] In fact,
\[
\int_{0}^{\pi}\left(\left|f_{1}(t)\right|^{2}+\left|f_{2}(t)\right|^{2}\right)dt=\int_{0}^{a}
\left(\left|\phi_{1}(t)\right|^{2}+\left|\phi_{2}(t)\right|^{2}\right)dt+
\]
\[
+\int_{a}^{\pi}\left(\left|\phi_{1}\left(\frac{t+\alpha a-a}{\alpha}\right)\right|^{2}+
\left|\phi_{2}\left(\frac{t+\alpha a-a}{\alpha}\right)\right|^{2}\right)dt
\]
\[
=\int_{0}^{a}
\left(\left|\phi_{1}(t)\right|^{2}+\left|\phi_{2}(t)\right|^{2}\right)dt+
\alpha\int_{a}^{\left(\frac{\pi+\alpha a-a}{\alpha}\right)}
\left(\left|\phi_{1}(t)\right|^{2}+\left|\phi_{2}(t)\right|^{2}\right)dt
\]
\[
\leq C\int_{0}^{\pi}\left(\left|\phi_{1}(t)\right|^{2}+\left|\phi_{2}(t)\right|^{2}\right)dt.
\]
Thus, the operator $L_{x}$ is invertible in $L_{2}(0,\pi).$ Therefore, the main equation (\ref{28}) is equivalent to
\[
A(x,.)+L_{x}^{-1}K_{x}A(x,.)=-L_{x}^{-1}F(x,.)
\]
and $L_{x}^{-1}K_{x}$ is completely continuous in $L_{2}(0,\pi)$. Then it is sufficient to prove that the equation
\begin{equation}\label{42}
g(\mu(t))+\int_{0}^{\mu(x)}g(\xi)F_{0}(\xi,t)d\xi=0
\end{equation}
has only trivial solution $g(t)=0.$ Let $g(t)$ be a non-trivial solution of (\ref{42}). Then
\[
\int_{0}^{x}\left(g_{1}^{2}(\mu(t))+g_{2}^{2}(\mu(t))\right)\rho(t)dt
+\int_{0}^{x}\int_{0}^{\mu(x)}\left(g(\xi)F_{0}(\xi,t),\ g(\mu(t))\right)\rho(t)d\xi dt=0.
\] It follows from (\ref{29}) that 
\[
\int_{0}^{x}\left(g_{1}^{2}(\mu(t))+g_{2}^{2}(\mu(t))\right)\rho(t)dt+
\]
\[
+\int_{0}^{x}\tilde{g}(\mu(t))\rho(t)\int_{0}^{\mu(x)}g(\xi)\left(\sum_{n=-\infty}^{\infty}
\frac{1}{\alpha_{n}}\left(\begin{array}{c}
\sin\lambda_{n}\xi \\
-\cos\lambda_{n}\xi
\end{array}
\right)\tilde{\varphi}_{0}(t,\lambda_{n})-\right.
\]
\[
\left.-\frac{1}{\mu(\pi)}\left(\begin{array}{c}
\sin\lambda_{n}^{0}\xi \\
-\cos\lambda_{n}^{0}\xi
\end{array}
\right)\tilde{\varphi}_{0}(t,\lambda_{n}^{0})\right)d\xi dt=0.
\] Using (\ref{36}), we get
\[
\int_{0}^{x}\left(g_{1}^{2}(\mu(t))+g_{2}^{2}(\mu(t))\right)\rho(t)dt+
\]
\[
+\int_{0}^{x}\tilde{g}(\mu(t))\rho(t)\int_{0}^{a}g(\xi)\sum_{n=-\infty}^{\infty}
\frac{1}{\alpha_{n}}\varphi_{0}(\xi,\lambda_{n})\tilde{\varphi}_{0}(t,\lambda_{n})d\xi dt-
\]
\[
-\int_{0}^{x}\tilde{g}(\mu(t))\rho(t)\int_{0}^{a}g(\xi)\sum_{n=-\infty}^{\infty}
\frac{1}{\mu(\pi)}\varphi_{0}(\xi,\lambda_{n}^{0})\tilde{\varphi}_{0}(t,\lambda_{n}^{0})d\xi dt+
\]
\[
+\int_{0}^{x}\tilde{g}(\mu(t))\rho(t)\int_{0}^{\alpha x-\alpha a+a}g(\xi)\sum_{n=-\infty}^{\infty}
\frac{1}{\alpha_{n}}\varphi_{0}\left(\frac{\xi}{\alpha}+a-\frac{a}{\alpha},
\lambda_{n}\right)\tilde{\varphi}_{0}(t,\lambda_{n})d\xi dt
\]
\[
-\int_{0}^{x}\tilde{g}(\mu(t))\rho(t)\int_{0}^{\alpha x-\alpha a+a}g(\xi)
\sum_{n=-\infty}^{\infty}
\frac{1}{\mu(\pi)}\varphi_{0}\left(\frac{\xi}{\alpha}+a-\frac{a}{\alpha},
\lambda_{n}^{0}\right)\tilde{\varphi}_{0}(t,\lambda_{n}^{0})d\xi dt=0.
\] Substituting $\frac{\xi}{\alpha}+a-\frac{a}{\alpha}\rightarrow\xi$ into the last two integrals, we obtain
\[
\int_{0}^{x}\left(g_{1}^{2}(\mu(t))+g_{2}^{2}(\mu(t))\right)\rho(t)dt+
\]
\[
+\int_{0}^{x}\tilde{g}(\mu(t))\rho(t)\int_{0}^{a}g(\xi)\sum_{n=-\infty}^{\infty}
\frac{1}{\alpha_{n}}\varphi_{0}(\xi,\lambda_{n})\tilde{\varphi}_{0}(t,\lambda_{n})d\xi dt-
\]
\[
-\int_{0}^{x}\tilde{g}(\mu(t))\rho(t)\int_{0}^{a}g(\xi)\sum_{n=-\infty}^{\infty}
\frac{1}{\mu(\pi)}\varphi_{0}(\xi,\lambda_{n}^{0})\tilde{\varphi}_{0}(t,\lambda_{n}^{0})d\xi dt+
\]
\[
+\alpha\int_{0}^{x}\tilde{g}(\mu(t))\rho(t)\int_{a}^{x}g(\alpha\xi-\alpha a+a)\sum_{n=-\infty}^{\infty}
\frac{1}{\alpha_{n}}\varphi_{0}(\xi,\lambda_{n})\tilde{\varphi}_{0}(t,\lambda_{n})d\xi dt-
\]
\[
-\alpha\int_{0}^{x}\tilde{g}(\mu(t))\rho(t)\int_{a}^{x}g(\alpha\xi-\alpha a+a)\sum_{n=-\infty}^{\infty}
\frac{1}{\mu(\pi)}\varphi_{0}(\xi,\lambda_{n}^{0})\tilde{\varphi}_{0}(t,\lambda_{n}^{0})d\xi dt=
\]
\[
=\int_{0}^{x}\left(g_{1}^{2}(\mu(t))+g_{2}^{2}(\mu(t))\right)\rho(t)dt+
\]
\[
+\int_{0}^{x}\tilde{g}(\mu(t))\rho(t)\int_{0}^{x}g(\mu(\xi))\rho(\xi)\sum_{n=-\infty}^{\infty}
\frac{1}{\alpha_{n}}\varphi_{0}(\xi,\lambda_{n})\tilde{\varphi}_{0}(t,\lambda_{n})d\xi dt-
\]
\begin{equation}\label{43}
-\int_{0}^{x}\tilde{g}(\mu(t))\rho(t)\int_{0}^{x}g(\mu(\xi))\rho(\xi)\sum_{n=-\infty}^{\infty}
\frac{1}{\mu(\pi)}\varphi_{0}(\xi,\lambda_{n}^{0})\tilde{\varphi}_{0}(t,\lambda_{n}^{0})d\xi dt=0.
\end{equation} 
Using Parseval equality, 
\[
g(\mu(t))=\sum_{n=-\infty}^{\infty}\left(\frac{1}{\mu(\pi)}
\int_{0}^{x}g(\mu(t))\tilde{\varphi}_{0}(t,\lambda_{n}^{0})\rho(t)dt\right)
\varphi_{0}(t,\lambda_{n}^{0}), 
\] it follows from (\ref{43}) that
\[
\int_{0}^{x}\tilde{g}(\mu(t))\rho(t)\int_{0}^{x}g(\mu(\xi))\rho(\xi)\sum_{n=-\infty}^{\infty}
\frac{1}{\alpha_{n}}\varphi_{0}(\xi,\lambda_{n})\tilde{\varphi}_{0}(t,\lambda_{n})d\xi dt=0.
\] Since the system $\left\{\varphi_{0}(t,\lambda_{n})\right\},$ $(n\in\mathbb{Z})$ 
is complete in $L_{2,\rho}(0,\pi;\mathbb{C}^{2})$, 
we have $g(\mu(t))\equiv 0$, i.e. $\left(L_{x}g\right)(t)=0$. For $L_{x}$ is invertible in 
$L_{2}(0,\pi)$, $A(x,.)=0$ is obtained. 
\end{proof}
\begin{theorem}
Let $L(\Omega(x),h_{1}, h_{2})$ and $\hat{L}(\hat{\Omega}(x),\hat{h}_{1},\hat{h}_{2})$ be two boundary value problems and 
\[
\lambda_{n}=\hat{\lambda}_{n}, \ \ \ \ \ \alpha_{n}=\hat{\alpha}_{n}, \ \ \left(n\in\mathbb{Z}\right).
\] Then 
\[
\Omega(x)=\hat{\Omega}(x) \ a.e. \ on \left(0,\pi\right), \ \ h_{1}=\hat{h}_{1}, \ \ h_{2}=\hat{h}_{2}.
\]
\end{theorem}
\begin{proof}
According to (\ref{29}) and (\ref{30}), $F_{0}(x,t)=\hat{F}_{0}(x,t)$ and 
$F(x,t)=\hat{F}(x,t)$. Then, from the fundamental
equation (\ref{28}), we have $A(x,t)=\hat{A}(x,t)$. It follows from (\ref{4}) that 
$\Omega(x)=\hat{\Omega}(x)$ a.e. on $\left(0,\pi\right)$. Taking into account (\ref{3}), we find
$\varphi(x,\lambda_{n})=\hat{\varphi}(x,\lambda_{n})$. In consideration of (\ref{a}), we get 
$\dot{\Delta}(\lambda_{n})=\dot{\hat{\Delta}}(\lambda_{n})$ and from (\ref{7}), $\beta_{n}=\hat{\beta}_{n}$.
Thus, using (\ref{b}) and (\ref{6}), we obtain $h_{1}=\hat{h}_{1}, \ \ h_{2}=\hat{h}_{2}$.
\end{proof}
\section{Solution of Inverse Problem}
\quad Let the real numbers $\left\{\lambda_{n},\alpha_{n}\right\}$, $(n\in\mathbb{Z})$ 
of the form (\ref{9}) and (\ref{10}) be given. Using these numbers, we construct the functions 
$F_{0}(x,t)$ and $F(x,t)$ by the formulas (\ref{29}) and (\ref{30}) and determine $A(x,t)$ 
from the fundamental equation (\ref{28}).

Now, let us construct the function $\varphi(x,\lambda)$ by the formula (\ref{3}), 
the function $\Omega(x)$ by the formula (\ref{4}), $\Delta(\lambda)$ by the formula (\ref{a}) and $\beta_{n}$
by the formula (\ref{7}) respectively, i.e.,
\[
\varphi(x,\lambda):=\varphi_{0}(x,\lambda)+\int_{0}^{\mu(x)}A(x,t)\left(\begin{array}{c}
\sin\lambda t \\
-\cos\lambda t
\end{array}\right)dt,
\]
\[
\Omega(x):=\rho(x)A(x,\mu(x))B-BA(x,\mu(x)),
\]
\[
\Delta(\lambda):=-\mu(\pi)(\lambda_{0}^{2}-\lambda^{2})\prod_{n=1}^{\infty}
\frac{\lambda_{n}^{2}-\lambda^{2}}{(\lambda_{n}^{0})^{2}},
\]
\[
\beta_{n}:=\frac{\dot{\Delta}(\lambda_{n})}{\alpha_{n}}\neq 0.
\]
The function $F_{0}(x,t)$ can be rewritten as follows:
\[
F_{0}(x,t)=\frac{1}{2}\left[a(x-\mu(t))+a(x+\mu(t))T\right],
\]
where
\[
a(x)=\sum_{n=-\infty}^{\infty}\left[\frac{1}{\alpha_{n}}\left(\begin{array}{cc}
\cos\lambda_{n}x & -\sin\lambda_{n}x \\
\sin\lambda_{n}x & \cos\lambda_{n}x
\end{array}
\right)-\frac{1}{\mu(\pi)}\left(\begin{array}{cc}
\cos\lambda_{n}^{0}x & -\sin\lambda_{n}^{0}x \\
\sin\lambda_{n}^{0}x & \cos\lambda_{n}^{0}x
\end{array}
\right)\right]
\] 
and $T=\left(\begin{array}{cc}
-1 & 0 \\
0 & 1
\end{array}\right).$ Analogously in \cite{Yur}, it is shown that the function $a(x)\in W_{2}^{1}[0,2\pi].$ 

\subsection{Derivation of the Differential Equation}
\begin{lemma} \label{lm6}
The relations hold: 
\begin{equation}\label{44}
B\varphi'(x,\lambda)+\Omega(x)\varphi(x,\lambda)=\lambda\rho(x)\varphi(x,\lambda),
\end{equation}
\begin{equation}\label{45}
\varphi_{1}(0,\lambda)=0, \ \ \ \ \varphi_{2}(0,\lambda)=-1.
\end{equation}
\end{lemma}
\begin{proof}
Differentiating on $x$ and $y$ the equation (\ref{28}) respectively, we get
\begin{equation}\label{46}
A_{x}^{'}(x,\mu(t))+F_{x}^{'}(x,t)+A(x,\mu(x))F_{0}(\mu(x),t)+\int_{0}^{\mu(x)}A_{x}^{'}(x,\xi)F_{0}(\xi,t)d\xi=0,
\end{equation} 
\begin{equation}\label{47}
\rho(t)A_{t}^{'}(x,\mu(t))+F_{t}^{'}(x,t)+\int_{0}^{\mu(x)}A(x,\xi)F_{0_{t}}^{'}(\xi,t)d\xi=0.
\end{equation}
It follows from (\ref{29}) and (\ref{30}) that
\begin{equation}\label{48}
\frac{\partial}{\partial t}F_{0}(x,t)B+\rho(t)B\frac{\partial}{\partial x}F_{0}(x,t)=0,
\end{equation}
\begin{equation}\label{49}
\rho(x)\frac{\partial}{\partial t}F(x,t)B+\rho(t)B\frac{\partial}{\partial x}F(x,t)=0.
\end{equation}
Since $F_{0}(x,0)BS=0$ and $F(x,0)BS=0$, where $S=\left(\begin{array}{c}
0 \\
-1
\end{array}\right)$,
using the main equation (\ref{28}), we obtain
\begin{equation}\label{50}
A(x,0)BS=0,
\end{equation}
or
\[
A_{11}(x,0)=A_{21}(x,0)=0.
\]
Multiplying the equation (\ref{46}) on the left by $B$ and $\rho(t)$ 
we get
\[
\rho(t)BF_{x}^{'}(x,t)+\rho(t)BA_{x}^{'}(x,\mu(t))+\rho(t)BA(x,\mu(x))F_{0}(\mu(x),t)+
\]
\begin{equation}\label{51}
+\rho(t)\int_{0}^{\mu(x)}BA_{x}^{'}(x,\xi)F_{0}(\xi,t)d\xi=0
\end{equation}
and multiplying the equation (\ref{47}) on the right by $B$ and $\rho(x)$ we have
\begin{equation}\label{52}
\rho(x)F_{t}^{'}(x,t)B+\rho(x)\rho(t)A_{t}^{'}(x,\mu(t))B+\rho(x)\int_{0}^{\mu(x)}A(x,\xi)F_{0_{t}}^{'}(\xi,t)Bd\xi=0.
\end{equation}
Adding (\ref{51}) and (\ref{52}) and using (\ref{49}), we find 
\[
\rho(t)BA_{x}^{'}(x,\mu(t))+\rho(t)BA(x,\mu(x))F_{0}(\mu(x),t)+
\rho(t)\int_{0}^{\mu(x)}BA_{x}^{'}(x,\xi)F_{0}(\xi,t)d\xi=
\]
\begin{equation}\label{53}
=-\rho(x)\rho(t)A_{t}^{'}(x,\mu(t))B-\rho(x)\int_{0}^{\mu(x)}A(x,\xi)F_{0_{t}}^{'}(\xi,t)B d\xi\equiv I(x,t).
\end{equation}
From (\ref{48}), we get
\begin{equation}\label{54}
I(x,t)=-\rho(x)\rho(t)A_{t}^{'}(x,\mu(t))B+\rho(x)\rho(t)\int_{0}^{\mu(x)}A(x,\xi)BF_{0_{\xi}}^{'}(\xi,t) d\xi.
\end{equation}
Integrating by parts and from (\ref{50})
\[
I(x,t)=-\rho(x)\rho(t)A_{t}^{'}(x,\mu(t))B+\rho(t)\rho(x)A(x,\mu(x))BF_{0}(\mu(x),t)-
\]
\begin{equation}\label{55}
-\rho(x)\rho(t)\int_{0}^{\mu(x)}A_{\xi}^{'}(x,\xi)BF_{0}(\xi,t)d\xi
\end{equation}
is obtained. Substituting (\ref{55}) into (\ref{53}) and divided by $\rho(t)\neq 0$, we have
\[
BA_{x}^{'}(x,\mu(t))+BA(x,\mu(x))F_{0}(\mu(x),t)
-\rho(x)A(x,\mu(x))BF_{0}(\mu(x),t)+
\]
\begin{equation}\label{56}
+\rho(x)A_{t}^{'}(x,\mu(t))B+\int_{0}^{\mu(x)}\left[BA_{x}^{'}(x,\xi)
+\rho(x)A_{\xi}^{'}(x,\xi)B\right]F_{0}(\xi,t)d\xi=0.
\end{equation}
Multiplying (\ref{28}) on the left by $\Omega(x)$ in the form of (\ref{4}) and add to (\ref{56})
\[
BA_{x}^{'}(x,\mu(t))+\rho(x)A_{t}^{'}(x,\mu(t))B+\Omega(x)A(x,\mu(t))+
\]
\begin{equation}\label{57}
+\int_{0}^{\mu(x)}\left[BA_{x}^{'}(x,\xi)+\rho(x)A_{\xi}^{'}(x,\xi)B+\Omega(x)A(x,\xi)\right]F_{0}(\xi,t)dt=0
\end{equation}
is obtained. Setting
\[
J(x,t):=BA_{x}^{'}(x,t)+\rho(x)A_{t}^{'}(x,t)B+\Omega(x)A(x,t),
\] we can rewrite equation (\ref{57}) as follows
\begin{equation}\label{58}
J(x,\mu(t))+\int_{0}^{\mu(x)}J(x,\xi)F_{0}(\xi,t)d\xi=0.
\end{equation}
According to Lemma \ref{thm4}, homogeneous equation (\ref{58}) 
has only the trivial solution, i.e.
\begin{equation}\label{59}
BA_{x}^{'}(x,t)+\rho(x)A_{t}^{'}(x,t)B+\Omega(x)A(x,t)=0, \ \ \ 0<t<x.
\end{equation}
Differentiating (\ref{3}) and multiplying on the left by $B$, we have
\[
B\varphi'(x,\lambda)=\lambda\rho(x)B\left(\begin{array}{c}
\cos\lambda\mu(x) \\
\sin\lambda\mu(x)
\end{array}\right)+BA(x,\mu(x))\left(\begin{array}{c}
\sin\lambda\mu(x) \\
-\cos\lambda\mu(x)
\end{array}\right)+
\]
\begin{equation}\label{60}
+\int_{0}^{\mu(x)}BA_{x}^{'}(x,t)\left(\begin{array}{c}
\sin\lambda t \\
-\cos\lambda t
\end{array}\right)dt.
\end{equation}
On the other hand, multiplying (\ref{3}) on the left by $\lambda\rho(x)$ and then integrating by parts and using 
(\ref{50}), we find
\[
\lambda\rho(x)\varphi(x,\lambda)=\lambda\rho(x)\left(\begin{array}{c}
\sin\lambda\mu(x) \\
-\cos\lambda\mu(x)
\end{array}\right)+\rho(x)A(x,\mu(x))B\left(\begin{array}{c}
\sin\lambda\mu(x) \\
-\cos\lambda\mu(x)
\end{array}\right)-
\]
\begin{equation}\label{61}
-\rho(x)\int_{0}^{\mu(x)}A_{t}^{'}(x,t)B\left(\begin{array}{c}
\sin\lambda t\\
-\cos\lambda t
\end{array}\right)dt.
\end{equation}
It follows from (\ref{60}) and (\ref{61}) that
\[
\lambda\rho(x)\varphi(x,\lambda)=B\varphi'(x,\lambda)-
\left[BA(x,\mu(x))-\rho(x)A(x,\mu(x))B\right]\left(\begin{array}{c}
\sin\lambda\mu(x)\\
-\cos\lambda\mu(x)
\end{array}\right)
\]
\[
-\int_{0}^{\mu(x)}\left[BA_{x}^{'}(x,t)+\rho(x)A_{t}^{'}(x,t)B\right]\left(\begin{array}{c}
\sin\lambda t\\
-\cos\lambda t
\end{array}\right)dt.
\] Taking into account (\ref{4}) and (\ref{59}),
\[
B\varphi'(x,\lambda)+\Omega(x)\varphi(x,\lambda)=\lambda\rho(x)\varphi(x,\lambda)
\] is obtained. For $x=0$, from (\ref{3}) we get (\ref{45}).
\end{proof}

\subsection{Derivation of Parseval Equality}
\begin{lemma} \label{lm7}
For each function $g(x)\in L_{2,\rho}(0,\pi;\mathbb{C}^{2})$, the following relation holds:
\begin{equation}\label{62}
\int_{0}^{\pi}\left(g_{1}^{2}(x)+g_{2}^{2}(x)\right)\rho(x)dx=\sum_{n=-\infty}^{\infty}
\frac{1}{\alpha_{n}}\left(\int_{0}^{\pi}\tilde{\varphi}(t,\lambda_{n})g(t)\rho(t)dt\right)^{2}.
\end{equation}
\end{lemma}
\begin{proof}
It follows from (\ref{3}) and (\ref{36}) that
\begin{equation}\label{63}
\varphi(x,\lambda)=\varphi_{0}(x,\lambda)+\int_{0}^{x}A(x,\mu(t))\varphi_{0}(t,\lambda)\rho(t)dt.
\end{equation}
Using the expression
\[
F_{0}(x,t)=\left\{\begin{array}{c}
F(x,t), \ \ \ \ \ \ \ \ \ \ \ \ \ \ x<a, \\
F\left(\frac{x}{\alpha}+a-\frac{a}{\alpha},t\right), \ \ x>a,
\end{array}\right.
\] main equation (\ref{28}) transforms into the following form
\begin{equation}\label{64}
A(x,\mu(t))+F(x,t)+\int_{0}^{x}A(x,\mu(\xi))F(\xi,t)\rho(\xi)d\xi=0.
\end{equation}
From the relation (\ref{63}), we get
\begin{equation}\label{65}
\varphi_{0}(x,\lambda)=\varphi(x,\lambda)+\int_{0}^{x}H(x,\mu(t))\varphi(t,\lambda)\rho(t)dt
\end{equation}
and for the kernel $H(x,\mu(t))$ we have the identity
\begin{equation}\label{66}
\tilde{H}(x,\mu(t))=F(t,x)+\int_{0}^{x}A(x,\mu(\xi))F(\xi,t)\rho(\xi)d\xi.
\end{equation}
Denote 
\[
Q(\lambda):=\int_{0}^{\pi}\tilde{\varphi}(t,\lambda)g(t)\rho(t)dt
\] and using (\ref{63}) transform into the following form
\[
Q(\lambda)=\int_{0}^{\pi}\tilde{\varphi}_{0}(t,\lambda)h(t)\rho(t)dt,
\] where
\begin{equation}\label{67}
h(t)=g(t)+\int_{t}^{\pi}\tilde{A}(s,\mu(t))g(s)\rho(s)ds.
\end{equation}
Similarly, in view of (\ref{65}), we have
\begin{equation}\label{68}
g(t)=h(t)+\int_{t}^{\pi}\tilde{H}(s,\mu(t))h(s)\rho(s)ds.
\end{equation}
According to (\ref{67}),
\[
\int_{0}^{\pi}F(x,t)h(t)\rho(t)dt=\int_{0}^{\pi}F(x,t)\left[g(t)+\int_{t}^{\pi}
\tilde{A}(s,\mu(t))g(s)\rho(s)ds\right]\rho(t)dt
\]
\[
=\int_{0}^{\pi}\left[F(x,t)+\int_{0}^{t}F(x,s)\tilde{A}(t,\mu(s))\rho(s)ds\right]g(t)\rho(t)dt
\]
\[
=\int_{0}^{x}\left[F(x,t)+\int_{0}^{t}F(x,s)\tilde{A}(t,\mu(s))\rho(s)ds\right]g(t)\rho(t)dt+
\]
\[
+\int_{x}^{\pi}\left[F(x,t)+\int_{0}^{t}F(x,s)\tilde{A}(t,\mu(s))\rho(s)ds\right]g(t)\rho(t)dt.
\]
It follows from (\ref{64}) and (\ref{66}) that 
\begin{equation}\label{69}
\int_{0}^{\pi}F(x,t)h(t)\rho(t)dt=\int_{0}^{x}H(x,\mu(t))g(t)\rho(t)dt-\int_{x}^{\pi}\tilde{A}(t,\mu(x))g(t)\rho(t)dt.
\end{equation}
From (\ref{33}) and Parseval equality we obtain,
\[
\int_{0}^{\pi}\left(h_{1}^{2}(t)+h_{2}^{2}(t)\right)\rho(t)dt+\int_{0}^{\pi}\tilde{h}(x)
F(x,t)h(t)\rho(t)\rho(x)dxdt=
\]
\begin{eqnarray}\nonumber
\lefteqn{=\int_{0}^{\pi}\left(h_{1}^{2}(t)+h_{2}^{2}(t)\right)\rho(t)dt+\sum_{n=-\infty}^{\infty}\frac{1}{\alpha_{n}}
\left(\int_{0}^{\pi}\tilde{\varphi}_{0}(t,\lambda_{n})h(t)\rho(t)dt\right)^{2}}
\\ \nonumber
&&-\sum_{n=-\infty}^{\infty}\frac{1}{\mu(\pi)}
\left(\int_{0}^{\pi}\tilde{\varphi}_{0}(t,\lambda_{n}^{0})h(t)\rho(t)dt\right)^{2} 
\\ \nonumber
&=&\sum_{n=-\infty}^{\infty}\frac{1}{\alpha_{n}}
\left(\int_{0}^{\pi}\tilde{\varphi}_{0}(t,\lambda_{n})h(t)\rho(t)dt\right)^{2}
=\sum_{n=-\infty}^{\infty}\frac{Q^{2}(\lambda_{n})}{\alpha_{n}}. \ \ \ \ \nonumber
\end{eqnarray}
Taking into account (\ref{69}), we have
\[
\sum_{n=-\infty}^{\infty}\frac{Q^{2}(\lambda_{n})}{\alpha_{n}}=
\int_{0}^{\pi}\left(h_{1}^{2}(t)+h_{2}^{2}(t)\right)\rho(t)dt+
\]
\[
+\int_{0}^{\pi}\tilde{h}(x)\left(\int_{0}^{x}H(x,\mu(t))g(t)\rho(t)dt\right)\rho(x)dx-
\]
\[
-\int_{0}^{\pi}\tilde{h}(x)\left(\int_{x}^{\pi}\tilde{A}(t,\mu(x))g(t)\rho(t)dt\right)\rho(x)dx
\]
\[
=\int_{0}^{\pi}\left(h_{1}^{2}(t)+h_{2}^{2}(t)\right)\rho(t)dt+\int_{0}^{\pi}\left(\int_{t}^{\pi}
\tilde{h}(x)H(x,\mu(t))\rho(x)dx\right)g(t)\rho(t)dt-
\]
\[
-\int_{0}^{\pi}\tilde{h}(x)\left(\int_{x}^{\pi}\tilde{A}(t,\mu(x))g(t)\rho(t)dt\right)\rho(x)dx,
\] whence by formulas (\ref{67}) and (\ref{68}), 
\begin{eqnarray}\nonumber
\lefteqn{\sum_{n=-\infty}^{\infty}\frac{Q^{2}(\lambda_{n})}{\alpha_{n}}=
\int_{0}^{\pi}\left(h_{1}^{2}(t)+h_{2}^{2}(t)\right)\rho(t)dt+}
\\ \nonumber
&&+\int_{0}^{\pi}\left(\tilde{g}(t)-\tilde{h}(t)\right)g(t)\rho(t)dt-
\int_{0}^{\pi}\tilde{h}(x)\left(h(x)-g(x)\right)\rho(x)dx
\\ \nonumber
&=&\int_{0}^{\pi}\left(g_{1}^{2}(t)+g_{2}^{2}(t)\right)\rho(t)dt \nonumber
\end{eqnarray}
is obtained, i.e., the relation (\ref{62}) is valid.
\end{proof}
\begin{corollary}
For any function $f(x)$ and $g(x)$ $\in L_{2,\rho}(0,\pi;\mathbb{C}^{2})$, the relation holds:
\begin{equation}\label{70}
\int_{0}^{\pi}\tilde{g}(x)f(x)\rho(x)dx=\sum_{n=-\infty}^{\infty}\frac{1}{\alpha_{n}}
\left(\int_{0}^{\pi}\tilde{g}(t)\varphi(t,\lambda_{n})\rho(t)dt\right)
\left(\int_{0}^{\pi}\tilde{\varphi}(t,\lambda_{n})f(t)\rho(t)dt\right).
\end{equation} 
\end{corollary}
\begin{lemma} \label{lm2}
For any $f(x)\in W_{2}^{1}[0,\pi]$, the expansion formula
\begin{equation}\label{71}
f(x)=\sum_{n=-\infty}^{\infty}c_{n}\varphi(x,\lambda_{n})
\end{equation}
is valid, where
\[
c_{n}=\frac{1}{\alpha_{n}}\int_{0}^{\pi}\tilde{\varphi}(x,\lambda_{n})f(x)\rho(x)dx.
\]
\end{lemma}
\begin{proof}
Consider the series 
\begin{equation}\label{72}
f^{*}(x)=\sum_{n=-\infty}^{\infty}c_{n}\varphi(x,\lambda_{n}),
\end{equation}
where
\begin{equation}\label{73}
c_{n}:=\frac{1}{\alpha_{n}}\int_{0}^{\pi}\tilde{\varphi}(x,\lambda_{n})f(x)\rho(x)dx.
\end{equation}
Using Lemma \ref{lm6} and integrating by parts , we get
\[
c_{n}=\frac{1}{\alpha_{n}\lambda_{n}}\int_{0}^{\pi}\left[-\frac{\partial}{\partial x}
\tilde{\varphi}(x,\lambda_{n})B+\tilde{\varphi}(x,\lambda_{n})\Omega(x)\right]f(x)dx=
\]
\[
=\frac{-1}{\alpha_{n}\lambda_{n}}\left[\tilde{\varphi}(\pi,\lambda_{n})Bf(\pi)-
\tilde{\varphi}(0,\lambda_{n})Bf(0)\right]+
\]
\[
+\frac{1}{\alpha_{n}\lambda_{n}}\int_{0}^{\pi}\tilde{\varphi}(x,\lambda_{n})
\left[Bf'(x)+\Omega(x)f(x)\right]dx.
\]
Applying the asymptotic formulas in Theorem \ref{thm1}, we find $\left\{c_{n}\right\}\in l_{2}.$ Consequently
the series (\ref{72}) converges absolutely and uniformly on $[0,\pi]$. According to (\ref{70}) and (\ref{73}), we have
\[
\int_{0}^{\pi}\tilde{g}(x)f(x)\rho(x)dx=\ \ \ \ \ \ \ \ \ \ \ \ \ \ \ \ \ \ \ \ \ \ \ \ \ \ \ \ \ \ \ \ \ \ \
\ \ \ \ \ \ \ \ \ \ \ \ \ \ \ \ \ \ \ \ \ \ \ \ \ \ \ \ \
\]
\[
=\sum_{n=-\infty}^{\infty}\frac{1}{\alpha_{n}}
\left(\int_{0}^{\pi}\tilde{g}(t)\varphi(t,\lambda_{n})\rho(t)dt\right)\left(\int_{0}^{\pi}
\tilde{\varphi}(t,\lambda_{n})f(t)\rho(t)dt\right) \ \ \ \ \ \ \ \ \ \ \ \ \ \ 
\]
\[
=\sum_{n=-\infty}^{\infty}c_{n}\left(\int_{0}^{\pi}\tilde{g}(t)\varphi(t,\lambda_{n})\rho(t)dt\right)
=\int_{0}^{\pi}\tilde{g}(t)\left(\sum_{n=-\infty}^{\infty}c_{n}\varphi(t,\lambda_{n})\right)\rho(t)dt
\]
\[
=\int_{0}^{\pi}\tilde{g}(t)f^{*}(t)\rho(t)dt. \ \ \ \ \ \ \ \ \ \ \ \ \ \ \ \ \ \ \ \ \ \ \ \ \ \ \ \ \ \ \ \ \ \ \
\ \ \ \ \ \ \ \ \ \ \ \ \ \ \ \ \ \ \ \ \ \ \ \ \ \ \ \ \ \ \ \
\] Since $g(x)$ is arbitrary, $f(x)=f^{*}(x)$ is obtained, i.e., the expansion formula (\ref{71}) is found.
\end{proof}

\subsection{Derivation of Boundary Condition}
\begin{lemma} \label{lm4}
The following equality holds:
\begin{equation}\label{74}
\sum_{n=-\infty}^{\infty}\frac{\varphi(x,\lambda_{n})}{\alpha_{n}\beta_{n}}=0.
\end{equation}
\end{lemma}
\begin{proof}
Using residue theorem, we get
\begin{equation}\label{75}
\sum_{n=-\infty}^{\infty}\frac{\varphi(x,\lambda_{n})}{\alpha_{n}\beta_{n}}=
\sum_{n=-\infty}^{\infty}\frac{\varphi(x,\lambda_{n})}{\dot{\Delta}(\lambda_{n})}=
\sum_{n=-\infty}^{\infty}\underset{\lambda=\lambda_{n}}{Res}\frac{\varphi(x,\lambda)}{\Delta(\lambda)}
=\frac{1}{2\pi i}\int_{\Gamma_{N}}\frac{\varphi(x,\lambda)}{\Delta(\lambda)}d\lambda,
\end{equation} 
where $\Gamma_{N}=\left\{\lambda:\left|\lambda\right|=\frac{N\pi}{\mu(\pi)}+\frac{\pi}{2\mu(\pi)}\right\}$. 
From (\ref{a}) and (\cite{Mar}, Lemma 3.4.2),

\begin{equation}\label{76}
\Delta(\lambda)=\lambda\sin\lambda\mu(\pi)+O(e^{\left|Im\lambda\right|\mu(\pi)}). 
\end{equation}
We denote $G_{\delta}=\left\{\lambda:\left|\lambda-\frac{n\pi}{\mu(\pi)}\right|\geq\delta, \ \ n=0,\pm 1,\pm 2...\right\}$
for some small fixed $\delta>0$ and $\left|\sin\lambda\mu(\pi)\right|\geq C_{\delta}e^{\left|Im\lambda\right|\mu(\pi)}$,
$\lambda\in G_{\delta}$, where $C_{\delta}$ positive number. Therefore, we have
\[
\left|\Delta(\lambda)\right|\geq C_{\delta}\left|\lambda\right|e^{\left|Im\lambda\right|\mu(\pi)}, \ \ \lambda\in G_{\delta}.
\] Using this inequality and (\ref{11}), we obtain (\ref{74}). 
\end{proof}
\begin{theorem} \label{thm7}
The following relation is valid:
\[
\left(\lambda_{n}+h_{1}\right)\varphi_{1}(\pi,\lambda_{n})+h_{2}\varphi_{2}(\pi,\lambda_{n})=0.
\]
\end{theorem}

\begin{proof}
From (\ref{74}), we can write for any $n_{0}\in\mathbb{Z}$
\begin{equation}\label{77}
\frac{\varphi(x,\lambda_{n_{0}})}{\alpha_{n_{0}}}=-\sum_{\stackrel{n=-\infty}{n\neq n_{0}}}^{\infty}
\frac{\beta_{n_{0}}\varphi(x,\lambda_{n})}{\alpha_{n}\beta_{n}}
\end{equation}
Let $m\neq n_{0}$ be any fixed number and $f(x)=\varphi(x,\lambda_{k})$. Then substituting (\ref{77}) in (\ref{71})
\[
\varphi(x,\lambda_{k})=\sum_{\stackrel{n=-\infty}{n\neq n_{0}}}^{\infty}c_{nk}\varphi(x,\lambda_{n}),
\] where
\[
c_{nk}=\frac{1}{\alpha_{n}}\int_{0}^{\pi}\left[\tilde{\varphi}(t,\lambda_{n})-\frac{\beta_{n_{0}}}{\beta_{n}}
\tilde{\varphi}(t,\lambda_{n_{0}})\right]\varphi(t,\lambda_{k})\rho(t)dt.
\] The system of functions $\{\varphi_{0}(x,\lambda_{n})\}$, $\left(n\in\mathbb{Z}\right)$ 
is orthogonal in $L_{2,\rho}(0,\pi;\mathbb{C}^{2})$.
Then by (\ref{3}), the system of functions $\{\varphi(x,\lambda_{n})\},$ $\left(n\in\mathbb{Z}\right)$ 
is orthogonal in $L_{2,\rho}(0,\pi;\mathbb{C}^{2})$ as well. 
Therefore, $c_{nk}=\delta_{nk},$ where $\delta_{nk}$ is Kronecker delta. Let us define 
\begin{equation}\label{78}
a_{nk}:=\int_{0}^{\pi}\tilde{\varphi}(t,\lambda_{n})\varphi(t,\lambda_{k})\rho(t)dt. 
\end{equation} 
Using this expression, we have for $n\neq k$
\begin{equation}\label{79}
a_{kk}-\frac{\beta_{n}}{\beta_{k}}a_{nk}=\alpha_{k}.
\end{equation}
It follows from (\ref{78}) that $a_{nk}=a_{kn}$. Taking into account this equality and (\ref{79}),  
\[
\beta_{k}^{2}\left(\alpha_{k}-a_{kk}\right)=\beta_{n}^{2}\left(\alpha_{n}-a_{nn}\right)=H, \ \ \ k\neq n,
\] where $H$ is a constant.
Then, we have
\[
\int_{0}^{\pi}\tilde{\varphi}(t,\lambda_{n})\varphi(t,\lambda_{n})\rho(t)dt=\alpha_{n}-\frac{H}{\beta_{n}^{2}}
\]
and
\[
\int_{0}^{\pi}\tilde{\varphi}(t,\lambda_{k})\varphi(t,\lambda_{n})\rho(t)dt=-\frac{H}{\beta_{k}\beta_{n}}, \ \ k\neq n.
\] 
It is easily obtained that for $k\neq n$, 
\[
\int_{0}^{\pi}\left[\varphi_{1}(x,\lambda_{k})\varphi_{1}(x,\lambda_{n})+
\varphi_{2}(x,\lambda_{k})\varphi_{2}(x,\lambda_{n})\right]\rho(x)dx=
\]
\[
=\frac{1}{(\lambda_{k}-\lambda_{n})}\left[\varphi_{2}(\pi,\lambda_{k})\varphi_{1}(\pi,\lambda_{n})
-\varphi_{1}(\pi,\lambda_{k})\varphi_{2}(\pi,\lambda_{n})\right]
=-\frac{H}{\beta_{k}\beta_{n}}.
\] According to the last equation, for $n\neq k,$
\begin{equation}\label{80}
\beta_{k}\varphi_{2}(\pi,\lambda_{k})\beta_{n}\varphi_{1}(\pi,\lambda_{n})
-\beta_{k}\varphi_{1}(\pi,\lambda_{k})\beta_{n}\varphi_{2}(\pi,\lambda_{n})=-H{(\lambda_{k}-\lambda_{n})}.
\end{equation}
We denote 
\begin{equation}\label{81}
D_{n}:=\beta_{n}\varphi_{1}(\pi,\lambda_{n}), \ \ \ \  
E_{n}:=\beta_{n}\varphi_{2}(\pi,\lambda_{n}).
\end{equation}
Then, we can rewrite equation (\ref{80}) as follows
\begin{equation}\label{82}
D_{k}E_{n}-E_{k}D_{n}=H{(\lambda_{k}-\lambda_{n})}, \ \ \ n\neq k.
\end{equation} 
Let $i,\ j,\ k, \ n$ be pairwise distinct integers, then we get
\[
\begin{array}{c}
D_{k}E_{n}-E_{k}D_{n}=H{(\lambda_{k}-\lambda_{n})}, \\
D_{n}E_{i}-E_{n}D_{i}=H{(\lambda_{n}-\lambda_{i})}, \\
D_{i}E_{k}-E_{i}D_{k}=H{(\lambda_{i}-\lambda_{k})}.
\end{array}
\] Adding them together, we find
\[
D_{n}(E_{i}-E_{k})+E_{n}(D_{k}-D_{i})=E_{i}D_{k}-D_{i}E_{k}.
\] In this equation, replacing $n$ by $j$, we get another equation
\[
D_{j}(E_{i}-E_{k})+E_{j}(D_{k}-D_{i})=E_{i}D_{k}-D_{i}E_{k}.
\] Subtracting the last two equation,
\[
(D_{n}-D_{j})(E_{i}-E_{k})=(D_{i}-D_{k})(E_{n}-E_{j}).
\] 
In the case of $E_{n}=E_{j}$, for some $n,\ j\in\mathbb{Z}$, then $E_{n}=$const. From (\ref{82}), 
$D_{n}=\kappa_{1}\lambda_{n}+\kappa_{2}$. In the case of $E_{n}\neq E_{j}$, then
we obtain $D_{n}=\kappa_{1}\lambda_{n}+\kappa_{2}$ and $E_{n}=\kappa_{3}\lambda_{n}+\kappa_{4}$, 
where in both cases $\kappa_{1},\ \kappa_{2},\ \kappa_{3},\ \kappa_{4}$ are constant.
Therefore, using these relation in (\ref{81}), we find
\[
\beta_{n}\varphi_{1}(\pi,\lambda_{n})=\kappa_{1}\lambda_{n}+\kappa_{2}, \ \ \
\beta_{n}\varphi_{2}(\pi,\lambda_{n})=\kappa_{3}\lambda_{n}+\kappa_{4}.
\]
Using 
\[
\varphi_{1}(\pi,\lambda_{n})=O\left(\frac{1}{n}\right), \ \ \ 
\varphi_{2}(\pi,\lambda_{n})=(-1)^{n+1}+O\left(\frac{1}{n}\right), 
\] $\lambda_{n}=\frac{n\pi}{\mu(\pi)}+O\left(\frac{1}{n}\right)$ and 
$\beta_{n}=\frac{n\pi}{\mu(\pi)}(-1)^{n}+O(1)$ derived from (\ref{7}) and (\ref{76}), 
we obtain $\kappa_{1}=0$, $\kappa_{3}=-1$. Denoting $h_{2}:=\kappa_{2}$ and $h_{1}:=-\kappa_{4}$, 
\[
h_{2}\varphi_{2}(\pi,\lambda_{n})=-\left(\lambda_{n}+h_{1}\right)\varphi_{1}(\pi,\lambda_{n}), \ \ n\in\mathbb{Z}
\] 
is obtained and it follows from (\ref{82}) that $H=h_{2}$.
\end{proof}

Thus, we have proved the following theorem:

\begin{theorem}
For the sequences $\left\{\lambda_{n}, \alpha_{n}\right\}$, $\left(n\in\mathbb{Z}\right)$, to be the spectral data for a
certain boundary value problem $L(\Omega(x), h_{1}, h_{2})$ of the form (\ref{1}), (\ref{2}) with $\Omega(x)\in L_{2}(0,\pi)$
it is necessary and sufficient that the relations (\ref{8}) and (\ref{10}) hold.
\end{theorem}

Algorithm of the construction of the function $\Omega(x)$ by spectral data $\left\{\lambda_{n},\alpha_{n}\right\}$,
$\left(n\in\mathbb{Z}\right)$ follows from the proof of the theorem: \\
1) By the given numbers $\left\{\lambda_{n},\alpha_{n}\right\}$,
$\left(n\in\mathbb{Z}\right)$ the functions $F_{0}(x,t)$ and $F(x,t)$ are 
respectively constructed by formula (\ref{29}) and (\ref{30}), \\
2) The function $A(x,t)$ is found from equation (\ref{28}), \\
3) $\Omega(x)$ is calculated by the formula (\ref{4}).

% ------------------------------------------------------------------------

\subsection*{Acknowledgment}
This work is supported by The Scientific and Technological Research Council of Turkey 
(T\"{U}B\.{I}TAK).

\end{document}